\documentclass{amsproc}

\usepackage{amsmath,amsthm,amsfonts,amscd,eucal}
\usepackage{graphicx}
\usepackage{epsfig}

\numberwithin{equation}{section}

\def\ca{{\mathcal A}}
\def\cb{{\mathcal B}}
\def\cc{{\mathcal C}}

\def\cf{{\mathcal F}}
\def\cg{{\mathcal G}}

\def\ck{{\mathcal K}}

\def\cn{{\mathcal N}}

\def\cp{{\mathcal P}}

\def\cu{{\mathcal U}}
\def\cv{{\mathcal V}}

\def\bc{{\mathbb C}}

\def\bn{{\mathbb N}}
\def\br{{\mathbb R}}
\def\bt{{\mathbb T}}
\def\bz{{\mathbb Z}}

\def\a{\alpha}

\def\g{\gamma}        \def\G{\Gamma}
        
\def\eps{\varepsilon}

\def\th{\vartheta}

\def\l{\lambda}       
\def\m{\mu}

\def\r{\rho}
\def\s{\sigma}        
     
\def\t{\tau}

\def\f{\varphi}

        \def\O{\Omega}

\def\itm#1{\item{$(#1)$}}
\newcommand{\set}[1]{\left\{#1\right\}}
\newcommand{\card}[1]{\left| \set{#1} \right|}
\newcommand{\ds}{\displaystyle}
\newcommand{\supp}{\text{supp}}
\newcommand{\conv}{\text{conv}\,}
\DeclareMathOperator{\Log}{Log}
\DeclareMathOperator{\arcosh}{arcosh}

\DeclareMathOperator{\Tr}{Tr}

\newcommand{\Det}{\text{det}}
\newcommand{\norm}[1]{\| #1 \|}

\newcommand{\Ci}{\cc}  
\newcommand{\Nt}{\cc^{\mathrm{notail}}}  
\newcommand{\Ta}{\cc^{\mathrm{tail}}}  
\renewcommand{\Pr}{\cp}  

\newcommand{\Complessi}{{\mathbb C}}
\newcommand{\Naturali}{{\mathbb N}}
\newcommand{\Interi}{{\mathbb Z}}

\newcommand{\DomFond}{{\mathcal F}} 
\newcommand{\ClosedPath}{{\mathcal C}}  
\newcommand{\NoTail}{{\mathcal C}^{\text{notail}}}  
\newcommand{\Tail}{{\mathcal C}^{\text{tail}}}  
\newcommand{\Cycle}{{\mathcal K}}  
\newcommand{\PrimeCycle}{{\mathcal P}}  

\newtheorem{Thm}{Theorem}[section]
\newtheorem{Cor}[Thm]{Corollary}
\newtheorem{Prop}[Thm]{Proposition}
\newtheorem{Lemma}[Thm]{Lemma}

\theoremstyle{definition}
\newtheorem{Dfn}[Thm]{Definition}
\newtheorem{exmp}[Thm]{Example}

\theoremstyle{remark}
\newtheorem{rem}[Thm]{Remark}

\begin{document}

\title{Zeta functions for infinite graphs \\
and functional equations}
\author{Daniele Guido, Tommaso Isola}
\address{Dipartimento di Matematica, Universit\`a di Roma ``Tor
Vergata'', I--00133 Roma, Italy.}
\email{guido@mat.uniroma2.it, isola@mat.uniroma2.it}
\thanks{The  authors were partially 
supported by GNAMPA, MIUR, the European Network ÒQuantum Spaces -
Noncommutative GeometryÓ HPRN-CT-2002-00280,  GDRE GREFI GENCO, and the ERC Advanced Grant 227458 OACFT } 


\begin{abstract}
  The definitions and main properties of the Ihara and Bartholdi zeta functions for infinite graphs are reviewed.  The general question of the validity of a functional equation  is discussed, and various possible solutions are proposed.
\end{abstract}
  


\maketitle

 \setcounter{section}{-1}

\section{Introduction}

In this paper, we review the main results concerning the Ihara zeta function and the Bartholdi zeta function for infinite graphs. Moreover, we propose various possible solutions to the problem of the validity of a functional equation for those zeta functions.

 The zeta function associated to a finite graph by Ihara, Sunada, Hashimoto
 and others, combines features of Riemann's zeta function, Artin
 L-functions, and Selberg's zeta function, and may be viewed as an
 analogue of the Dedekind zeta function of a number field \cite{Bass,HaHo,Hashi,Ihara,KoSu,StTe,Sunada}.  It is
 defined by an Euler product over proper primitive cycles of the graph. 
 
 A main result for the Ihara zeta function $Z_X(z)$ associated with a graph $X$, is the so called determinant formula, which shows that the inverse of this function can be written, up to a polynomial, as $\Det(I-Az+Qz^{2})$, where $A$ is the adjacency matrix and $Q$ is the diagonal matrix corresponding to the degree minus 1. As a consequence, for a finite graph, $Z_X(z)$ is indeed the inverse of a polynomial, hence can be extended meromorphically to the whole plane.
 
 A second main result is the fact that, for $(q+1)$-regular graphs, namely graphs with  degree constantly equal to $(q+1)$, $Z_X$, or better its so called completion $\xi_X$, satisfies a functional equation, namely is invariant under the transformation $z\to \frac1{qz}$.
 
 The first of the mentioned results has been proved for infinite (periodic or fractal) graphs  in \cite{GILa01,GILa03}, by  introducing the analytic determinant for operator algebras. For first results and discussions about the functional equation we still refer to \cite{GILa01,GILa03} and to \cite{Clair}.
 
 The Bartholdi zeta function $Z_{X}(z,u)$ was introduced by Bartholdi  in \cite{Ba1} as a two-variable generalization of the Ihara zeta function.
Such function coincides with the Ihara zeta function for $u=0$, and gives the Euler product on all primitive cycles for $u=1$. Bartholdi also showed that some results for the Ihara zeta function extend to this new zeta function. We quote \cite{Choe,KLS,MiSu1,MiSu5} for further results and generalizations of the Bartholdi zeta function. The extension to the case of infinite periodic simple graphs is contained in \cite{GILa04}, where a functional equation for regular graphs and a determinant formula are proved. Sato \cite{Sa2} generalised  the determinant formula  to the non simple case, and also proved it for the case of fractal graphs  \cite{Sa1}.

 The aim of this paper  is two-fold: on the one hand we illustrate all the mentioned results both for the periodic and the fractal case, using  a unified approach in all the statements and also in some proofs, while for others we only treat the fractal case, referring the readers to \cite{GILa04} for the periodic case. On the other hand, we analyze the meaning and validity of the functional equation for infinite graphs.
 
Let us recall that the functional equation may be seen as a simple corollary of the determinant formula, which can be written in such a way that the argument of the determinant is itself invariant under the desired transformation of the complex plane.
However, for infinite graphs, the determinant is no longer a polynomial, and its zeroes are no longer isolated. As a consequence, the singularities of the Ihara zeta may constitute a barrier to the possibility of extending it analitically to an unbounded domain. In particular, for $(q+1)$-regular graphs, the singularities are contained in the curve $\Omega_q$ which disconnects the plane,
       \begin{figure}[ht]
    \centering
    \psfig{file=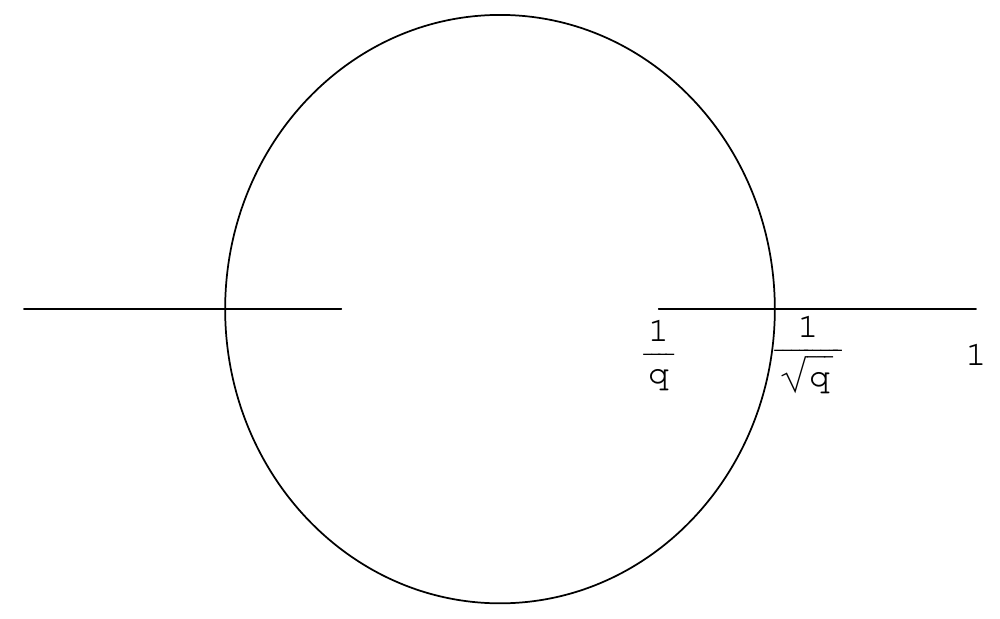,height=1.5in}
    \caption{The set $\Omega_q$}
    \label{fig:Omega}
      \end{figure}
hence may confine $Z_X$ to the bounded component of $\Omega_q^c$. In this case the functional equation  loses its meaning: one may still use the determinant formula to define $Z_X$ in the unbounded component of $\Omega_q^c$, but in this case the functional equation is not a theorem but a definition.

A first solution, due to Clair \cite{Clair}, consists in the observation that in some cases the Ihara zeta naturally extends to a holomorphic function on a branched covering of the complex plane. There,  the functional equation holds if we extend the transformation $z\to1/qz$ of $\bc$ to a transformation of the covering which interchanges the two branches.

Another solution, presented here, shows that a small amount of analyticity of the distribution function $F$ given by the trace of the spectral function of the adjacency operator is sufficient to provide a suitable analytic extension of the Ihara zeta function, which satisfies the functional equation.

As we shall see, the problem of the analytic extension does not arise for the Bartholdi zeta function on infinite graphs.
Indeed, in the case of $(q+1)$-regular graphs, namely graphs with  degree constantly equal to $(q+1)$, the determinant formula takes the following simpler form:
\begin{equation*}
 Z_{X}(z,u) = (1-(1-u)^2z^2)^{-(q-1)/2}
    \big( \Det_{\t}((1+(1-u)(q+u)z^2)I-zA)\big)^{-1} .
\end{equation*}
While the left-hand side is defined only in a suitable neighborhood of the origin in $\bc^2$, the right-hand side is a holomorphic function on an open set whose complement is always contained in a three-dimensional real submanifold  $\Omega$ of $\bc^2$ containing all possible  singularities.
We show that the complement of $\Omega$ is connected, hence the Bartholdi zeta function holomorphically extends to $\Omega^c$. Moreover, it satisfies a functional equation on such domain.

We shall use such result on the Bartholdi zeta to give a third solution to the analytic extension problem for the Ihara zeta, which now works in full generality, and satisfies the functional equation. 
The procedure is  the following: add a variable to the Ihara zeta so to get the Bartholdi zeta $Z_{X}(z,u)$, extend it holomorphically to $\Omega^c$ and then set $u=0$. Such function is the desired extension of the Ihara zeta $Z_{X}(z)$ to $\Omega_q^c$, and satisfies the functional equation.


\section{Zeta functions for infinite graphs}
\subsection{Preliminaries}

 In this section, we recall some terminology from graph theory, and 
 introduce the class of geometric operators on an infinite graph. 

 A {\it simple graph} $X=(VX,EX)$ is a collection $VX$ of objects,
 called {\it vertices}, and a collection $EX$ of unordered pairs of
 distinct vertices, called {\it edges}.  The edge $e=\set{u,v}$ is
 said to join the vertices $u,v$, while $u$ and $v$ are said to be
 {\it adjacent}, which is denoted $u\sim v$.  A {\it path} (of length
 $m$) in $X$ from $v_0\in VX$ to $v_m\in VX$, is
 $(v_{0},\ldots,v_{m})$, where $v_{i}\in VX$, $v_{i+1}\sim v_i$, for
 $i=0,...,m-1$ (note that $m$ is the number of edges in the path).  A
 path is {\it closed} if $v_{m}=v_{0}$.
 
 We assume that $X$ is countable and connected, $i.e.$ there is a path
 between any pair of distinct vertices.  Denote by $\deg(v)$ the degree
 of $v\in VX$, $i.e.$ the number of vertices adjacent to $v$.  We
 assume that $X$ has bounded degree, $i.e.$ $d := \sup_{v\in VX}
 \deg(v) <\infty$.  Denote by $\r$ the combinatorial distance on $VX$,
 that is, for $v,w\in VX$, $\r(v,w)$ is the length of the shortest
 path between $v$ and $w$.  If $\Omega \subset VX$, $r\in\bn$, we
 write $B_r(\O) := \cup_{v\in\O} B_r(v)$, where $B_{r}(v) := \set{
 v'\in VX: \r(v',v)\leq r}$.

Recall that the {\it adjacency matrix} of $X$,
 $A=\big(A(v,w)\big)_{v,w\in VX}$, and the {\it degree matrix} of $X$,
 $D=\big(D(v,w)\big)_{v,w\in VX}$ are defined by
 \begin{equation}\label{eq:adjacency}
     A(v,w)=
     \begin{cases} 
	 1&v\sim w\\
	 0&\text{otherwise} 
     \end{cases}
 \end{equation}
 and
 \begin{equation}\label{eq:degree}
     D(v,w)=
     \begin{cases} 
	 \deg(v)&v= w\\
	 0&\text{otherwise.} 
     \end{cases}
 \end{equation}
 
 Then, considered as an operator on $\ell^2(VX)$, $\|A\|\leq d:=\sup_{v\in VX}
 \deg(v) <\infty$,  (see \cite{Mohar}, \cite{MoWo}). 

\subsubsection{Periodic graphs}
In this section, we introduce the classes of periodic graphs and operators, see \cite{GILa02,GILa03} for more details.

 Let $\Gamma$ be a countable discrete
 subgroup of automorphisms of $X$, which acts freely on $X$ ($i.e.$
 any $\gamma\in\Gamma$, $\gamma\neq id$ doesn't have fixed points), and with
 finite quotient $B:=X/\Gamma$ (observe that $B$ needn't be a simple graph). Denote by $\DomFond\subset VX$ a set of
 representatives for $VX/\Gamma$, the vertices of the quotient graph $B$.
 Let us define a unitary representation of $\Gamma$ on $\ell^{2}(VX)$ by
 $(\lambda(\gamma)f)(x):= f(\gamma^{-1}x)$, for $\gamma\in\Gamma$, $f\in\ell^{2}(VX)$,
 $x\in V(X)$.  Then the von Neumann algebra ${\mathcal N}(X,\Gamma):= \{ \lambda(\gamma) :
 \gamma\in\Gamma\}'$ of bounded operators on $\ell^{2}(VX)$ commuting with the
 action of $\Gamma$ inherits a trace given by $\Tr_{\Gamma}(T) = \sum_{x\in\DomFond}
 T(x,x)$, for $T\in {\mathcal N}(X,\Gamma)$.

  It is easy to see that $A,D\in{\mathcal N}(X,\Gamma)$.

\subsubsection{Self-similar graphs}\label{sec:selfSimilar}

 In this section, we introduce the class of self-similar graphs and the geometric operators over them (see \cite{GILa01} for more details).  This
 class contains many examples of what are usually called fractal
 graphs, see $e.g.$ \cite{Barl,HaKu}.

 If $K$ is a subgraph of $X$, we call {\it frontier} of $K$, and
 denote by $\cf(K)$, the family of vertices in $V K$ having distance 1
 from the complement of $V K$ in $VX$.

 \begin{Dfn}[Local Isomorphisms]
     A {\it local isomorphism} of the graph  $X$ is a triple
     \begin{equation}
	 \bigl(S(\gamma)\, ,R(\gamma)\, ,\gamma \bigr),
     \end{equation}
     where $S(\gamma)\, ,R(\gamma)$ are subgraphs of $X$ and
     $\gamma : S(\gamma)\to R(\gamma)$ is a graph isomorphism. 
\end{Dfn}

 \begin{Dfn}[Amenable graphs]
     A countably infinite graph with bounded degree $X$ is {\it
     amenable} if it has an {\it amenable exhaustion}, namely, an
     increasing family of finite subgraphs $\{K_n : n\in \mathbb{N}\}$
     such that $\cup_{n\in\bn} K_n = X$ and
    \begin{equation*}
	\frac{|\cf(K_n)|}{|K_n|}\to 0\qquad {\rm as}\,\,\, n\to
	\infty\, ,
    \end{equation*}
    where $|K_{n}|$ stands for $|VK_{n}|$ and $|\cdot|$ denotes the 
    cardinality.
\end{Dfn}

\begin{Dfn}[Self-similar graphs] \label{def:Quasiperiodic} 
    A countably infinite graph with bounded degree $X$ is
    {\it self-similar} if it has an amenable exhaustion $\{K_{n}\}$
    such that the following  conditions  $(i)$ and $(ii)$ hold: 
    
    \itm{i} For every $n\in\bn$, there is a finite set of local
    isomorphisms $\cg(n,n+1)$ such that, for all $\gamma\in
    \cg(n,n+1)$, one has $S(\gamma) = K_n$,
    \begin{equation}
	\bigcup_{\gamma\in \cg(n,n+1)} \gamma(K_n) = K_{n+1},
    \end{equation}
    and moreover, if $\gamma , \gamma^\prime\in \cg(n,n+1)$ with
    $\gamma\neq\gamma^\prime$,
    \begin{equation}
	V( \gamma K_n ) \cap V( \gamma' K_n ) = 
	\cf(\gamma K_n ) \cap \cf( \gamma' K_n ).
    \end{equation}
    
    \itm{ii} We then define $\cg(n,m)$, for $n<m$, as the set of all
    admissible products $\g_{m-1}\cdot\dots\cdot\g_{n}$, $\g_{i}\in
    \cg(i,i+1)$, where ``admissible'' means that, for each term of the
    product, the range of $\g_{j}$ is contained in the source of
    $\g_{j+1}$.  We also let $\cg(n,n)$ consist of the identity
    isomorphism on $K_{n}$, and $\cg(n):=\cup_{m\geq n}\cg(n,m)$.  We
    can now define the $\cg$-{\it invariant frontier} of $K_{n}$:
    $$
    \cf_{\cg}(K_{n})=
    \bigcup_{\gamma\in \cg(n)}\g^{-1}\cf( \gamma K_{n} ),
    $$
    and we require that
    \begin{equation}\label{strongreg}
	\frac{|\cf_{\cg}(K_n)|}{|K_n|}\to 0\qquad {\rm
	as}\; n\to \infty\, .
    \end{equation}
\end{Dfn}

 In the rest of the paper, we denote by $\cg$ the family of all
 local isomorphisms which can be written as (admissible) products
 $\g_1^{\eps_1} \g_2^{\eps_2}...\g_k^{\eps_k}$, where
 $\g_i\in\cup_{n\in\bn} \cg(n)$, $\eps_i\in\set{-1,1}$, for
 $i=1,...,k$ and $k\in\bn$.
  
 We refer to \cite{GILa01} for several examples of self-similar graphs.
  

\subsubsection{The C$^*$-algebra of geometric operators}

 \begin{Dfn}[Finite propagation operators]
    A bounded linear operator $A$ on $\ell^2 (VX)$ has {\it finite propagation} $r=r(A)\geq 0$ if, for all $v\in VX$, we have $\supp (Av)\subset B_{r}(v)$ and $\supp (A^{*}v)\subset B_{r}(v)$, where we use $v$ to mean the function which is $1$ on the vertex $v$ and $0$ otherwise, and $A^*$ is the Hilbert space adjoint of $A$.
 \end{Dfn}
 
 \begin{Dfn}[Geometric Operators]
     A {\it local isomorphism} $\gamma$ of the graph  $X$   defines a {\it partial
     isometry} $U(\gamma) : \ell^2 (VX)\to \ell^2 (VX)$, by setting
     \begin{align*}
	 U(\g)(v):=
	 \begin{cases}
	     \g(v)& v\in V(S(\g)) \\
	     0& v\not\in V(S(\g)), 
	 \end{cases}
     \end{align*}
     and extending by linearity.  A bounded operator $T$ acting on
     $\ell^2 (VX)$ is called {\it geometric} if there exists $r\in\bn$
     such that $T$ has finite propagation $r$ and, for any local
     isomorphism $\gamma$, any $v\in VX$ such that $B_{r}(v)\subset
     S(\g)$ and $B_{r}(\g v)\subset R(\g)$, one has
     \begin{equation}\label{eq:geometric}
       TU(\gamma)v = U(\gamma)Tv,\quad T^*U(\gamma)v = U(\gamma)T^*v\, .
     \end{equation}     
 \end{Dfn}

 \begin{Prop}\label{Prop:3.7}
     Geometric operators form a $^*$-algebra containing the adjacency
     operator $A$ and the degree operator $D$.
 \end{Prop}
 
\begin{Thm}\label{thm:trace}
    Let $X$ be a self-similar graph, and let $\ca (X)$ be the
    C$^*$-algebra defined as the norm closure of the $^*$-algebra of
    geometric operators.  Then, on $\ca(X)$, there is a well-defined
    trace state $\Tr_{\cg}$ given by
    \begin{equation}
	\Tr_{\cg} (T) = \lim_n \frac{\Tr\bigl( P(K_n) T \bigr)}
	{\Tr\bigl( P(K_n) \bigr)},
    \end{equation}
    where $P(K_n)$ is the orthogonal projection of $\ell^{2}(VX)$ onto
    its closed subspace $\ell^{2}(V K_n)$.
\end{Thm}

\subsection{Combinatorial results}

 The Bartholdi zeta function is defined by means of equivalence
 classes of primitive cycles.  Therefore, we need to introduce some
 terminology from graph theory, following \cite{StTe} with some
 modifications.

 \begin{Dfn}[Types of closed paths]
    \label{def:redPath}
    \itm{i} A path $C=(v_{0},\ldots,v_{m})$ in $X$ has {\it backtracking} if
    $v_{i-1}=v_{i+1}$, for some $i\in\{1,\ldots,m-1\}$.  We also say
    that $C$ has a bump at $v_{i}$.  Then, the bump count $bc(C)$ of
    $C$ is the number of bumps in $C$.  Moreover, if $C$ is a closed
    path of length $m$, the cyclic bump count is $cbc(C):= |\set{i\in\Interi_{m}:
    v_{i-1}=v_{i+1}}|$, where the indices are considered in $\Interi_{m}$,
    and $\Interi_{m}$ is the cyclic group on $m$ elements.

    \itm{ii} A  closed path is {\it primitive} if it is not obtained
    by going $k\geq 2$ times around some other closed path.

    \itm{iii} A closed path $C=(v_{0},\ldots,v_{m}=v_{0})$ has a {\it tail}
    if there is $k\in\{1,\ldots,[m/2]-1\}$ such that $v_{j}=v_{m-j}$, for
    $j=1,\ldots,k$.  Denote by $\ClosedPath$
 the set of closed paths, by $\Tail$ the set of closed paths with
    tail, and by $\NoTail$ the set of tail-less closed paths.  Observe
    that $\ClosedPath=\Tail\cup\NoTail$, $\Tail\cap\NoTail=\emptyset$.
 \end{Dfn}

 For any $m\in\Naturali$, $u\in\Complessi$, let us denote by $A_{m}(u)(x,y):=
 \sum_{P} u^{bc(P)}$, where the (finite) sum is over all paths $P$ in $X$, of
 length $m$, with initial vertex $x$ and terminal vertex $y$, for
 $x,y\in VX$.  Then $A_{1}=A$.  Let $A_{0}:= I$ and $Q:= D-I$.  Finally, let
 ${\mathcal U}\subset\Complessi$ be a bounded set containing $\set{0,1}$, and denote
 by $M({\mathcal U}):= \sup_{u\in{\mathcal U}} \max\set{|u|,|1-u|}\geq 1$, and
 $\alpha({\mathcal U}):= \frac{d+\sqrt{d^{2}+4M({\mathcal U})(d-1+M({\mathcal U}))}}{2}$.  

\begin{rem}
In the sequel, in order to unify the notation, we will denote by $(\cb(X),\t)$ the pair $(\cn(X,\Gamma),\Tr_\Gamma)$, or $(\ca(X),\Tr_\cg)$, as the case may be. Moreover, 
$$
\sideset{}{^*}\sum_{x\in X} f(x) = \begin{cases}
\sum \limits_{x\in \DomFond} f(x),  &	\textrm{if $X$ is a periodic graph} \\
\lim \limits_{n\to\infty} \frac{1}{|K_{n}|} \sum \limits_{x\in K_{n}}f(x), & \textrm{if $X$ is a self-similar graph},
\end{cases}
$$
denotes a mean on the graph. Of course, in the self-similar case, the limit must be shown to exist.
\end{rem}

 \begin{Lemma}\label{lem:Lemma1}
     \itm{i} $A_{2}(u) = A^{2}-(1-u)(Q+I)\in \cb(X)$,

     \itm{ii} for $m\geq 3$, $A_{m}(u) =
     A_{m-1}(u)A - (1-u)A_{m-2}(Q+uI) \in \cb(X)$,

     \itm{iii} $\sup_{u\in{\mathcal U}} \|A_{m}(u)\| \leq \alpha({\mathcal U})^{m}$, for $m\geq0$.
 \end{Lemma}
 \begin{proof}
     $(i)$  If $x = y$, then $A_{2}(u)(x,x)=\deg(x) u = (Q+I)(x,x) u$
     because there are $\deg(x)$ closed paths of length $2$ starting
     at $x$, whereas $A^{2}(x,x) = \deg(x) = (Q+I)(x,x)$, so that
     $A_{2}(u)(x,x)=A^{2}(x,x)-(1-u)(Q+I)(x,x)$.  If $x\neq y$, then
     $A^{2}(x,y)$ is the number of paths of length $2$ from $x$ to
     $y$, so $A_{2}(u)(x,y)=A^{2}(x,y) = A^{2}(x,y)-(1-u)(Q+I)(x,y)$.

     $(ii)$ For $x,y\in VX$, consider all the paths
     $P=(v_{0},\ldots,v_{m})$ of length $m$, with $v_{0}=x$ and
     $v_{m}=y$.  They can also be considered as obtained from a path
     $P'$ of length $m-2$ going from $x\equiv v_{0}$ to $v_{m-2}$,
     followed by a path of length $2$ from $v_{m-2}$ to $y\equiv
     v_{m}$.  There are four types of such paths: $(a)$ those $P$ for
     which $y\equiv v_{m}\neq v_{m-2}$, $v_{m-1}\neq v_{m-3}$, so that $bc(P)=bc(P')$; 
     $(b)$ those $P$ for
     which $y\equiv v_{m}\neq v_{m-2}$, $v_{m-1} = v_{m-3}$, so that $bc(P)=bc(P')+1$;
     $(c)$
     those $P$ for which $y\equiv v_{m}=v_{m-2}$, but $v_{m-1}\neq
     v_{m-3}$, so that $bc(P)=bc(P')+1$; $(d)$ those $P$ for which
     $y\equiv v_{m} = v_{m-2}$ and $v_{m-1}=v_{m-3}$, so that
     $bc(P)=bc(P')+2$.

     Therefore, the terms corresponding to those four types in
     $A_{m}(u)(x,y)$ are $u^{bc(P')}$, $u^{bc(P')+1}$, $u^{bc(P')+1}$, and $u^{bc(P')+2}$, respectively.

     On the other hand, the sum $\sum_{z\in VX} A_{m-1}(u)(x,z)A(z,y)$
     assigns, to those four types, respectively the values
     $u^{bc(P')}$, $u^{bc(P')+1}$, $u^{bc(P')}$, and $u^{bc(P')+1}$. Hence we need to introduce corrections for paths of types $(c)$ and $(d)$.

     Therefore $A_{m}(u)(x,y) = \sum_{z\in VX}
     A_{m-1}(u)(x,z)A(z,y)   + A_{m-2}(u)(x,y) (\deg(y)-1) (u-1) +
     A_{m-2}(u)(x,y) (u^{2}-u)$, where the second summand takes into account paths of type $(c)$, and the third is for paths of type $(d)$. The statement follows.

     $(iii)$ We have $\|A_{1}(u)\|=\|A\|\leq d \leq \alpha({\mathcal U})$, $\|A_{2}(u)\|\leq
     d^{2}+M({\mathcal U})d \leq \alpha({\mathcal U})^{2}$, and $\|A_{m}(u)\|\leq
     d\|A_{m-1}(u)\|+M({\mathcal U})(d-1+M({\mathcal U}))\|A_{m-2}(u)\|$, from which the
     claim follows by induction.
 \end{proof}

We now want to count the closed paths of length $m$ which have a tail.

 \begin{Lemma}\label{lem:countTail}
     For $m\in\bn$, let 
	$$
     t_{m}(u):= \sideset{}{^*}\sum_{x\in X} \sum_{C=(x,\ldots)\in\Tail_{m}}
     u^{bc(C)}
	$$
     Then 
     
     \itm{i} in the self-similar case, the above mean exists and is finite,
    
     \itm{ii} $t_{1}(u)=0$, $t_{2}(u)=u \t(Q+I)$, $t_3(u)=0$,
     
     \itm{iii} for $m\geq 4$, $t_{m}(u) =
     \t \bigl( (Q-(1-2u)I)A_{m-2}(u) \bigr) + (1-u)^{2}t_{m-2}(u)$,
       

     \itm{iv} for any $m\in\bn$, 
	\begin{multline*}
      t_{m}(u) = \t\Bigl(
     (Q-(1-2u)I)\sum_{j=1}^{[\frac{m-1}{2}]} (1-u)^{2j-2}A_{m-2j}(u)
     \Bigr) \\
      + \delta_{even}(m)u(1-u)^{m-2}\t(Q+I),
     \end{multline*}
     where
     $\delta_{even}(m)=\begin{cases} 1& m \text{ is even}\\
     0& m \text{ is odd.}\end{cases}$

     \itm{v} 
     $\sup_{u\in{\mathcal U}} |t_{m}(u)| \leq 4m\alpha({\mathcal U})^{m}$.

 \end{Lemma}
 \begin{proof}
	We consider only the case of self-similar graphs, for the periodic case see \cite{GILa04}.
     Denote by $(C,v)$ the closed path $C$ with the origin in
     $v\in VX$.
 
     $(i)$ For $n\in\bn$, $n>m$, let 
     $$
     \O_n:= V(K_n)\setminus B_m(\cf_\cg(K_n)), \qquad \O_{n}':=
     V(K_{n}) \cap B_{m}(\cf_{\cg}(K_{n})).
     $$ 
     Then, for all $p\in\bn$,
     $$
     V(K_{n+p}) = \biggl( \bigcup_{\g\in\cg(n,n+p)} \g\O_n \biggr) \cup
     \biggl( \bigcup_{\g\in\cg(n,n+p)} \g\O_{n}' \biggr).
     $$
     Let $t_m(x,u) := \sum_{C=(x,\ldots)\in\Tail_{m}} u^{bc(C)}$ so that $|t_m(x,u)| \leq d^{m-2}M(\cu)^{m-1}$.  Then
     \begin{align*}
	& \left| \frac{1}{|K_{n+p}|} \sum_{x\in K_{n+p}} t_m(x,u) -
	\frac{1}{|K_{n}|} \sum_{x\in K_{n}} t_m(x,u) \right| \\
	& \leq \left| \frac{|\cg(n,n+p)|}{|K_{n+p}|} \sum_{x\in
	\O_{n}} t_m(x,u) - \frac{1}{|K_{n}|} \sum_{x\in K_{n}} t_m(x,u)
	\right| + \frac{1}{|K_{n+p}|} \sum_{\g \in \cg(n,n+p)} \sum_{x\in \O_{n}'} |t_m(\g x,u)| \\
	& \leq \left| \frac{|\cg(n,n+p)|}{|K_{n+p}|} -
	\frac{1}{|K_{n}|} \right| \sum_{x\in K_{n}} |t_m(x,u)| + 
	\frac{|\cg(n,n+p)|}{|K_{n+p}|} \sum_{x\in B_m(\cf_\cg(K_n))}
	|t_m(x,u)| \\
	& \quad + \frac{1}{|K_{n+p}|} \sum_{\g \in \cg(n,n+p)} \sum_{x\in \O_{n}'} |t_m(\g x,u)| \\
	& \leq \left| 1- \frac{|K_{n}| |\cg(n,n+p)|}{|K_{n+p}|}
	\right| d^{m-2}M(\cu)^{m-1} + 2 \frac{|K_{n}| |\cg(n,n+p)|}{|K_{n+p}|}
	\frac{|B_m(\cf_\cg(K_n))|}{|K_n|} d^{m-2}M(\cu)^{m-1} \\
	& \leq 6d^{m-2}(d+1)^mM(\cu)^{m-1} \eps_n \to 0,\quad \text { as }
	n\to\infty,
    \end{align*}
    where, in the last inequality, we used \cite{GILa01} equations (3.2),
    (3.8) [with $r=1$], and the fact that $\eps_n =\displaystyle
    \frac{|\cf_\cg(K_n)|}{|K_n|}\to0$.
 
    $(ii)$ is easy to prove.
    
    $(iii)$ Let us define $\O:= \set{v\in VX: v\not\in K_n,
    \r(v,K_n)=1} \subset B_1(\cf_\cg(K_n))$.  We have
    \begin{align*}
	\frac{1}{|K_{n}|} \sum_{x\in K_{n}} \sum_{y\sim x}
	& \sum_{C=(x,y,\ldots)\in\Tail_{m}} u^{bc(C)} =\\
	&= \frac{1}{|K_{n}|} \sum_{y\in K_{n}} \sum_{x\sim y}
	 \sum_{C=(x,y,\ldots)\in\Tail_{m}} u^{bc(C)} \\
	& \quad + \frac{1}{|K_{n}|} \sum_{y\in \O} \sum_{x\in K_n,
	x\sim y}  \sum_{C=(x,y,\ldots)\in\Tail_{m}} u^{bc(C)} \\
	& \quad - \frac{1}{|K_{n}|} \sum_{y\in K_{n}} \sum_{x\in \O,
	x\sim y}  \sum_{C=(x,y,\ldots)\in\Tail_{m}} u^{bc(C)} .
    \end{align*}
    Since 
    $$
    \frac{1}{|K_{n}|} \sum_{y\in\O} \sum_{x\in K_n, x\sim y}
     \sum_{C=(x,y,\ldots)\in\Tail_{m}} |u^{bc(C)}| \leq \frac{1}{|K_n|}
    |\cf_\cg(K_n)|(d+1)d^{m-2}M(\cu)^{m-1} \to0
    $$
    and
    \begin{align*}
	 \frac{1}{|K_{n}|} \sum_{y\in K_{n}} \sum_{x\in \O, x\sim y}
	&\sum_{C=(x,y,\ldots)\in\Tail_{m}} |u^{bc(C)}| = \\
	& = \frac{1}{|K_{n}|} \sum_{y\in \cf_\cg(K_{n})} \sum_{x\in
	\O, x\sim y} \sum_{C=(x,y,\ldots)\in\Tail_{m}} |u^{bc(C)}| \\
	& \leq \frac{1}{|K_n|} |\cf_    \cg(K_n)|d^{m-2}M(\cu)^{m-1} \to0,
   \end{align*}
   we obtain
   \begin{align*}
       t_{m} &= \lim_{n\to\infty} \frac{1}{|K_{n}|} \sum_{x\in K_{n}}
       \sum_{C=(x,\ldots)\in\Tail_{m}} u^{bc(C)} \\
       &= \lim_{n\to\infty} \frac{1}{|K_{n}|} \sum_{x\in K_{n}}
       \sum_{y\sim x} \sum_{C=(x,y,\ldots)\in\Tail_{m}} u^{bc(C)} \\
       &= \lim_{n\to\infty} \frac{1}{|K_{n}|} \sum_{y\in K_{n}}
       \sum_{x\sim y} \sum_{C=(x,y,\ldots)\in\Tail_{m}} u^{bc(C)} .
   \end{align*}
     
A path $C$ in the last set goes from $x$ to $y$, then over a closed path $D=(y,v_{1},\ldots,v_{m-3},y)$ of length $m-2$, and then back to $x$.  There are two kinds of closed paths $D$ at $y$: those with tails and those without.

$Case\ 1:$ $D$ does not have a tail. \\ 

Then $C$ can be of two types: $(a)$ $C_{1}$, where $x\neq v_{1}$ and $x\neq v_{m-3}$; $(b)$ $C_{2}$, where $x=v_{1}$ or $x=v_{m-3}$. Hence, $bc(C_{1})=bc(D)$, and $bc(C_{2})=bc(D)+1$, and there are $\deg(y)-2$ possibilities for $x$ to be adjacent to $y$ in $C_{1}$, and $2$ possibilities in $C_{2}$.

$Case\ 2:$ $D$ has a tail.\\
     
Then $C$ can be of two types: $(c)$ $C_{3}$, where $v_{1}=v_{m-3}\neq x$; $(d)$ $C_{4}$, where $v_{1}=v_{m-3} = x$. Hence, $bc(C_{3})=bc(D)$, and $bc(C_{4})=bc(D)+2$, and there are $\deg(y)-1$ possibilities for $x$ to be adjacent to $y$ in $C_{3}$, and $1$ possibility in $C_{4}$.

    \noindent Therefore,
     \begin{align*}
     \sum_{x\sim y}
     & \sum_{C=(x,y,\ldots)\in\Tail_{m}} u^{bc(C)}  \\
     & = (\deg(y)-2) \sum_{D=(y,\ldots)\in\NoTail_{m-2}} u^{bc(D)} +2u
     \sum_{D=(y,\ldots)\in\NoTail_{m-2}} u^{bc(D)} \\
     & \qquad +(\deg(y)-1) \sum_{D=(y,\ldots)\in\Tail_{m-2}} u^{bc(D)} +
     u^{2} \sum_{D=(y,\ldots)\in\Tail_{m-2}} u^{bc(D)} \\
     & = (\deg(y)-2+2u) \sum_{D=(y,\ldots)\in\ClosedPath_{m-2}} u^{bc(D)}
     + (1-2u+u^{2}) \sum_{D=(y,\ldots)\in\Tail_{m-2}} u^{bc(D)},
     \end{align*}
     so that
     \begin{align*}
     t_{m}(u) &=  \lim_{n\to\infty} \frac{1}{|K_{n}|} \sum_{y\in K_{n}} \Bigl( (Q(y,y)-1+2u)\cdot A_{m-2}(u)(y,y) \\
     & \qquad + (1-u)^{2} 
     \sum_{D=(y,\ldots)\in\Tail_{m-2}} u^{bc(D)} \Bigr) \\
     & = \Tr_{\cg} \bigl( (Q-(1-2u)I)A_{m-2}(u) \bigr)+(1-u)^{2}t_{m-2}(u).
     \end{align*}
     $(iv)$ Follows from $(iii)$, and the fact that $\Tr_{\cg}((Q-(1-2u)I)A)=0$.
     
     $(v)$ Let us first observe that $M({\mathcal U}) < \alpha({\mathcal U})$, so that,
     from $(iv)$ we obtain, with $\alpha:=\alpha({\mathcal U}),\ M:= M({\mathcal U})$,
     \begin{align*}
     |t_{m}(u)| & \leq \norm{ Q-(1-2u)I } \,
     \sum_{j=1}^{[\frac{m-1}{2}]} |1-u|^{2j-2} \norm{ A_{m-2j}(u) }
      + |u||1-u|^{m-2} d  \\
     & \leq (d-2+2M) \sum_{j=1}^{[\frac{m-1}{2}]} M^{2j-2} \alpha^{m-2j} +
      M^{m-1} d  \\
     & \leq (d-2+2M) \Bigl[\frac{m-1}{2} \Bigr] \alpha^{m-2} +
      M^{m-1} d  \\
     & \leq  \Bigl(\Bigl[\frac{m-1}{2} \Bigr] 3\alpha^{m-1} + \alpha^{m} \Bigr) \leq
     4m\alpha^{m}.
     \end{align*}
 \end{proof}

 \begin{Lemma}\label{lem:estim.for.N}
     Let us define
	$$
	N_{m}(u) :=
    \sideset{}{^*} \sum_{x\in X} \sum_{C=(x,\ldots)\in\ClosedPath_{m}} u^{cbc(C)}.
	$$
	Then, for all $m\in\bn$,

	\itm{i}  in the self-similar case, the above mean exists and is finite,
     
     \itm{ii} $N_{m}(u) = \t(A_{m}(u)) -(1-u) t_{m}$,
     
     \itm{iii} $|N_{m}(u)| \leq Km\alpha({\mathcal U})^{m+1}$, where $K>0$ is independent of $m$.
 \end{Lemma}
 \begin{proof}
	We consider only the case of self-similar graphs, for the periodic case see \cite{GILa04}.

\itm{i}   the existence of $\lim_{n\to\infty} \frac{1}{|K_{n}|} \sum_{x\in K_{n}} \sum_{(C,x)\in\cc_{m}} u^{cbc(C)}$ can be proved as in Lemma
     \ref{lem:countTail} $(i)$.  
     
     \itm{ii} Therefore,
     \begin{align*}
	 N_{m}(u) & = \lim_{n\to\infty} \frac{1}{|K_{n}|} \sum_{x\in K_{n}} \sum_{(C,x)\in\cc_{m}} u^{cbc(C)} \\
	 & = \lim_{n\to\infty} \frac{1}{|K_{n}|} \sum_{x\in K_{n}}\biggl( \sum_{(C,x)\in\Nt_{m}} u^{bc(C)} + \sum_{(C,x)\in\Ta_{m}} u^{bc(C)+1} \biggr)\\
	 & = \lim_{n\to\infty} \frac{1}{|K_{n}|} \sum_{x\in K_{n}}\biggl( \sum_{(C,x)\in\cc_{m}} u^{bc(C)} + (u-1)\sum_{(C,x)\in\Ta_{m}} u^{bc(C)} \biggr)\\
	 & = \lim_{n\to\infty} \frac{1}{|K_{n}|}\,  \sum_{x\in
	 K_{n}} A_{m}(u)(x,x) +(u-1) \lim_{n\to\infty} \frac{1}{|K_{n}|} \sum_{x\in K_{n}}
     \sum_{C=(x,\ldots)\in\Tail_{m}} u^{bc(C)} \\
	 & = \Tr_{\cg}(A_{m}(u)) +(u-1) t_{m}.	 
     \end{align*}
     \itm{iii} This follows from $(ii)$. 
 \end{proof} 

\begin{rem}
Observe that in the self-similar case we can also write
     \begin{equation}\label{eq:Nm}
	 N_{m}(u) = \lim_{n\to\infty} \frac{1}{|K_{n}|}  \sum_{\substack{ C\in \cc_{m} \\
     C\subset K_{n} }} u^{cbc(C)}.
     \end{equation}
     Indeed, 
     \begin{align*}
	 0&\leq \frac{1}{|K_{n}|} \Biggl| \sum_{x\in K_{n}} \sum_{(C,x)\in\cc_{m}} u^{cbc(C)} - \sum_{\substack{ C\in \cc_{m} \\
     C\subset K_{n} }} u^{cbc(C)}  \Biggr| \\
	 & \leq \frac{1}{|K_{n}|}\, \sum_{\substack{ (C,x) \in \cc_{m}, C\not\subset K_{n} \\
     x\in K_{n} }} |u^{cbc(C)}|
	 \\
	 & \leq \frac{1}{|K_{n}|}\, \card{ (C,x)\in\Ci_{m}: x\in
	 B_m(\cf_\cg(K_n)) } M(\cu)^m \\
	 & = \frac{M(\cu)^m}{|K_{n}|}\, \sum_{x\in B_m(\cf_\cg(K_n))}
	 A_{m}(1)(x,x) = \frac{M(\cu)^m}{|K_{n}|}\, \Tr \bigl( P( B_m(\cf_\cg(K_n)))
	 A_{m}(1) \bigr) \\
	 & \leq M(\cu)^m \norm{A_{m}(1)} \, \frac{|B_m(\cf_\cg(K_n))|}{|K_{n}|} 
	 \leq M(\cu)^m \a(\cu)^m\, (d+1)^m \frac{| \cf_\cg(K_n)|}{|K_{n}|}\to
	 0,\, \text{as } n\to\infty.
     \end{align*}	
\end{rem}

\subsection{The Zeta function}\label{sec:Zeta}

In this section, we define the Bartholdi zeta function for a periodic graph and for a self-similar graph, and prove that it is a holomorphic function in a suitable open set.  In the rest of this work, ${\mathcal U}\subset\Complessi$ will denote a bounded open set containing $\set{0,1}$.

  \begin{Dfn}[Cycles]
    We say that two closed paths $C=(v_{0},\ldots,v_{m}=v_{0})$ and
    $D=(w_{0},\ldots,w_{m}=w_{0})$ are {\it equivalent}, and write
    $C\sim_{o} D$, if there is an integer $k$ such that
    $w_{j}=v_{j+k}$, for all $j$, where the addition is taken modulo
    $m$, that is, the origin of $D$ is shifted $k$ steps with respect
    to the origin of $C$.  The equivalence class of $C$ is denoted
    $[C]_o$.  An equivalence class is also called a {\it cycle}.
    Therefore, a closed path is just a cycle with a specified origin.

    Denote by $\ck$ the set of cycles, and by $\Pr\subset\ck$
    the subset of primitive cycles.
 \end{Dfn}

 \begin{Dfn}[Equivalence relation between cycles]
    Given $C$, $D\in\ck$, we say that $C$ and $D$ are $\cg$-{\it
    equivalent}, and write $C \sim_{\cg} D$, if there is a local
    isomorphism $\g\in \cg$ such that $D=\g(C)$.  We denote by
    $[\ck]_{\cg}$ the set of $\cg$-equivalence classes of
    cycles, and analogously for the subset $\Pr$. The notion of $\G$-equivalence is analogous (see \cite{GILa04} for details), and we denote by $[\cdot]_{\cg}$ also a $\G$-equivalence class.
 \end{Dfn}
 
  We recall from \cite{GILa01} and \cite{GILa04} several quantities associated to a cycle.
 
  \begin{Dfn}
     Let $C\in\ck$, and call 
     
     \itm{i} {\it effective length} of $C$, denoted $\ell(C)\in\bn$,
     the length of the primitive cycle $D$ underlying $C$, $i.e.$ such
     that $C=D^k$, for some $k\in\bn$, whereas the length of $C$ is denoted by $|C|$,

	\itm{ii} if $C$ is contained in a periodic graph, {\it stabilizer}  of $C$ in $\Gamma$ the subgroup $\Gamma_{C}= \{\gamma\in\Gamma :
 \gamma(C)=C\}$, whose order divides $\ell(C)$,    

     \itm{ii} if $C$ is contained in a self-similar graph, {\it size} of $C$, denoted $s(C)\in\bn$, the least $m\in\bn$
     such that $C\subset \g(K_{m})$, for some local isomorphism $\g\in
     \cg(m)$,
     
     \itm{iii} {\it average multiplicity} of $C$, the number in $[0,\infty)$ given by
     $$\mu(C) := \begin{cases}
	\frac1{|\G_C|}, & \textrm{if $C$ is contained in a periodic graph,} \\
	 \lim\limits_{n\to\infty} \frac{|\cg(s(C),n)|}{|K_{n}|}, & \textrm{if $C$ is contained in a self-similar graph.}
\end{cases}
     $$
	
 \end{Dfn}

That the limit actually exists is the content of the following
 
 \begin{Prop}
	Let $(X,\cg)$ be a self-similar graph.
     \itm{i} Let $C\in\ck$, then the following limit exists and is
     finite:
     $$
     \lim_{n}\frac{ |\cg(s(C),n)| }{|K_n|},
     $$

     \itm{ii} $s(C)$, $\ell(C)$, and $\m(C)$ only depend on
     $[C]_{\cg}\in [\ck]_{\cg}$; moreover, if $C=D^{k}$ for some
     $D\in\Pr$, $k\in\bn$, then $s(C)=s(D)$, $\ell(C)=\ell(D)$,
     $\m(C)=\m(D)$.
  \end{Prop}
 \begin{proof}
      See \cite{GILa01} Proposition 6.4.
  \end{proof}

\begin{Prop} 
	For $m\in\bn$, $\displaystyle{N_m(u) =
     \sum_{[C]_\cg\in[\ck_m]_\cg} \m(C)\ell(C)} u^{cbc(C)}$,\\
     where, as above, the subscript $m$ corresponds to cycles of length $m$.
\end{Prop}
\begin{proof}
We prove only the self-similar case, for the periodic case see \cite{GILa04}. We have successively:
     \begin{align*} 
	 N_m(u) & = \lim_{n\to\infty} \frac{1}{|K_{n}|}  \sum_{\substack{
	 C\in\cc_{m}\\ C\subset K_{n} }} u^{cbc(C)} \\
	 & = \lim_{n\to\infty} \sum_{[C]_\cg\in[\ck_m]_\cg}
	 \frac{1}{|K_{n}|} \ell(C)\, \sum_{ \substack{ D\in\ck_{m}, D\sim_\cg C\\ D\subset
	 K_{n} }} u^{cbc(D)} \\
	 & = \lim_{n\to\infty} \sum_{[C]_\cg\in[\ck_m]_\cg}
	 \frac{1}{|K_{n}|} \, \ell(C)\, | \cg(s(C),n)|  u^{cbc(C)} \\
	 & = \sum_{[C]_\cg\in[\ck_m]_\cg} \m(C)\ell(C) u^{cbc(C)},
     \end{align*} 
     where, in the last equality, we used dominated convergence.  
\end{proof}

 \begin{Dfn}[Zeta function]
     $$
     Z_{X}(z,u) := \prod_{[C]_{\cg}\in [\PrimeCycle]_{\cg}}
     (1-z^{|C|}u^{cbc(C)})^{ -\mu(C) }, \qquad
     z,u\in\Complessi.
     $$
 \end{Dfn}

 \begin{Prop}\label{lem:power.series}
     \itm{i} $Z_X(z,u):=\prod_{[C] \in [\PrimeCycle]_{\cg}}
     (1-z^{|C|}u^{cbc(C)})^{ -\mu(C) }$ defines a
     holomorphic function in $\set{(z,u)\in\Complessi^{2}:
     |z|<\frac{1}{\alpha({\mathcal U})}, u\in{\mathcal U} }$,

     \itm{ii} $z\frac{\partial_{z}Z_X(z,u)}{Z_X(z,u)} = \sum_{m=1}^{\infty}
     N_{m}(u)z^{m}$, where $N_{m}(u)$ is defined in Lemma
     \ref{lem:estim.for.N},

     \itm{iii} $Z_X(z,u) = \exp\left( \sum_{m=1}^{\infty}
     \frac{N_{m}(u)}{m}z^{m} \right)$.
 \end{Prop}
 \begin{proof}
     Let us observe that, for any $u\in{\mathcal U}$, and $z\in\Complessi$ such that
     $|z|<\frac{1}{\alpha({\mathcal U})}$,
     \begin{align*}
     \sum_{m=1}^{\infty} N_{m}(u)  z^{m} & =
     \sum_{m=1}^{\infty} \sum_{[C]_{\cg}\in [\Cycle_m]_{\cg}} \mu(C)\ell(C)u^{cbc(C)}\, z^{m} \\
     & =  \sum_{[C]_{\cg}\in [\Cycle]_{\cg}} \mu(C)\ell(C)u^{cbc(C)}\, z^{|C|} \\
     & =   \sum_{[C]_{\cg}\in
     [\PrimeCycle]_{\cg}} \sum_{m=1}^{\infty} \mu(C) |C|\,  u^{cbc(C^{m})}z^{|C^{m}|} \\
     &= \sum_{[C]_{\cg}\in [\PrimeCycle]_{\cg}} \mu(C) \,
     \sum_{m=1}^{\infty} |C| z^{|C|m} u^{cbc(C)m}\\
     &= \sum_{[C]_{\cg}\in [\PrimeCycle]_{\cg}} \mu(C) \,
     z\frac{\partial}{\partial z}
     \sum_{m=1}^{\infty} \frac{z^{|C|m}u^{cbc(C)m}}{m} \\
     &= -\sum_{[C]_{\cg}\in [\PrimeCycle]_{\cg}} \mu(C) \,
     z\frac{\partial}{\partial z} \log(1-z^{|C|}u^{cbc(C)}) \\
     & = z\frac{\partial}{\partial z} \log Z_X(z,u),
     \end{align*}
     where, in the last equality we used uniform convergence on
     compact subsets of $\set{(z,u)\in\Complessi^{2}: u\in{\mathcal U},
     |z|<\frac{1}{\alpha({\mathcal U})}}$.  The proof of the remaining statements
     is now clear.
 \end{proof}



\subsection{The determinant formula}\label{sec:DetFormula}

 In this section, we prove the main result in the theory of Bartholdi
 zeta functions, which says that  the reciprocal of $Z$ is, up to a factor, the determinant of a
 deformed Laplacian on the graph.  We first need some technical
 results.  Let us recall that $d:=\sup_{v\in VX} \deg(v)$,
 ${\mathcal U}\subset\Complessi$ is a bounded open set containing $\set{0,1}$,
 $M({\mathcal U}):= \sup_{u\in{\mathcal U}} \max\set{|u|,|1-u|}$, and $\alpha\equiv\alpha({\mathcal U}):=
 \frac{d+\sqrt{d^{2}+4M({\mathcal U})(d-1+M({\mathcal U}))}}{2}$.

 \begin{Lemma}\label{lem:eq.for.A}
 	For any $u\in{\mathcal U},\ |z|<\frac{1}{\alpha}$, one has
     \itm{i} $\left(\sum_{m\geq 0}
     A_{m}(u)z^{m}\right) \bigl(I-Az+(1-u)(Q+uI)z^{2} \bigr) =
     (1-(1-u)^{2}z^{2})I$, 

     \itm{ii} $\left(\sum_{m\geq 0} \left( \sum_{k=0}^{[m/2]}
     (1-u)^{2k} A_{m-2k}(u) \right)
     z^{m}\right) \bigl(I-Az+(1-u)(Q+uI)z^{2} \bigr) = I$.
 \end{Lemma}
 \begin{proof}
     $(i)$ From Lemma \ref{lem:Lemma1}, we obtain that
     \begin{align*}
     \biggl(\sum_{m\geq 0} &A_{m}(u) z^{m}\biggr) \bigl(I-Az+(1-u)(Q+uI)z^{2} \bigr)
     \\
     & = \sum_{m\geq 0} A_{m}(u)z^{m} - \sum_{m\geq 0}
     A_{m}(u)Az^{m+1} + \sum_{m\geq 0} (1-u)A_{m}(u)(Q+uI)z^{m+2}
     \\
     &= A_{0}(u)+A_{1}(u)z+A_{2}(u)z^{2}+ \sum_{m\geq 3}
     A_{m}(u)z^{m} \\
     & \qquad -A_{0}(u)Az -A_{1}(u)Az^{2} - \sum_{m\geq 3}
     A_{m-1}(u)Az^{m}\\
     & \qquad +(1-u)A_{0}(u)(Q+uI)z^{2}+ \sum_{m\geq 3}
     (1-u)A_{m-2}(Q+uI)z^{m} \\
     &= I+Az+\bigl( A^{2}-(1-u)(Q+I) \bigr)z^{2} -Az -A^{2}z^{2}
     +(1-u)(Q+uI)z^{2} \\
     & = (1-(1-u)^{2}z^{2})I.
     \end{align*}
     $(ii)$
     \begin{align*}
     I &= (1-(1-u)^{2}z^{2})^{-1} \biggl(\sum_{m\geq 0}
     A_{m}(u)z^{m}\biggr) \bigl(I-Az+(1-u)(Q+uI)z^{2} \bigr) \\
     &= \biggl(\sum_{m\geq 0} A_{m}(u)z^{m}\biggr) \biggl(
     \sum_{j=0}^{\infty}(1-u)^{2j}z^{2j}\biggr) \bigl(I-Az+(1-u)(Q+uI)z^{2} \bigr) \\
     &= \biggl(\sum_{k\geq 0}\sum_{j=0}^{\infty}
     A_{k}(u)(1-u)^{2j}z^{k+2j}\biggr) \bigl(I-Az+(1-u)(Q+uI)z^{2} \bigr) \\
     &= \biggl(\sum_{m\geq 0}\biggl( \sum_{j=0}^{[m/2]}
     A_{m-2j}(u)(1-u)^{2j}\biggr) z^{m}\biggr) \bigl(I-Az+(1-u)(Q+uI)z^{2} \bigr).
     \end{align*}
 \end{proof}

 \begin{Lemma}\label{lem:eq.for.B}
     Define
     $$\begin{cases}
	B_0(u):= I, \\
	B_{1}(u):=A, \\
	B_{m}(u) := A_{m}(u) - (Q-(1-2u)I) \sum_{k=1}^{[m/2]}
     (1-u)^{2k-1} A_{m-2k}(u), & m\geq 2.
	\end{cases}$$
%
%
     Then

     \itm{i} $B_{m}(u) \in \cb(X)$,

     \itm{ii} $B_{m}(u) = A_{m}(u) +(1-u)^{-1} \bigl( Q-(1-2u)I \bigr) A_{m}(u) -
     \bigl( Q-(1-2u)I \bigr) \sum_{k=0}^{[m/2]} (1-u)^{2k-1}A_{m-2k}(u)$,

     \itm{iii} 
     $$
    \t(B_{m}(u)) =
     \begin{cases}
     N_{m}(u) - (1-u)^{m} \t(Q-I) & m \text{ even} \\
     N_{m}(u)  & m \text{ odd,}
     \end{cases}
     $$

     \itm{iv}
     $$
     \sum_{m\geq 1} B_{m}(u)z^{m} = \left(
     Au-2(1-u)(Q+uI)z^{2}\right)\left(I-Az+(1-u)(Q+uI)z^{2}\right)^{-1}, \
     u\in{\mathcal U},\ |z|<\frac{1}{\alpha}.
     $$
 \end{Lemma}
 \begin{proof}
     $(i)$ and $(ii)$ follow from computations involving bounded
     operators.

     $(iii)$ It follows from Lemma \ref{lem:countTail} $(ii)$ that, if
     $m$ is odd,
     $$
     \t(B_{m}(u)) = \t(A_{m}(u)) - (1-u)t_{m}(u) = N_{m}(u),
     $$
     whereas, if $m$ is even,
     \begin{align*}
     \t(B_{m}(u)) & = \t(A_{m}(u)) -(1-u)^{m-1}
     \t(Q-(1-2u)I) \\
     & \qquad - (1-u)t_{m}(u) + (1-u)^{m-1}u \t(Q+I) \\
     &= N_{m}(u) - (1-u)^{m}\t(Q-I).
     \end{align*}

     $(iv)$ Using $(ii)$ we obtain
     \begin{align*}
     \biggl( \sum_{m\geq 0} &B_{m}(u)z^{m} \biggr) (I-Az+(1-u)(Q+uI)z^{2})
     \\
     & = \biggl( \bigl(I+(1-u)^{-1}(Q-(1-2u)I) \bigr) \sum_{m\geq
     0} A_{m}(u)z^{m} \\
     & \qquad
     - (1-u)^{-1}(Q-(1-2u)I) \sum_{m\geq
     0}\sum_{j=0}^{[m/2]} A_{m-2j}(u)(1-u)^{2j}z^{m}\biggr) (I-Az+(1-u)(Q+uI)z^{2}) \\
     \intertext{ (by Lemma \ref{lem:eq.for.A}) } & =
     \bigl(I+(1-u)^{-1}(Q-(1-2u)I) \bigr) (1-(1-u)^{2}z^{2})I -
     (1-u)^{-1}(Q-(1-2u)I)\\
     & = (1-(1-u)^{2}z^{2})I - (1-u)(Q-(1-2u)I)z^{2}.
     \end{align*}
     Since $B_{0}(u)=I$, we get
     \begin{align*}
     \biggl( \sum_{m\geq 1} & B_{m}(u)z^{m} \biggr)
     (I-Az+(1-u)(Q+uI)z^{2}) \\
     &= (1-(1-u)^{2}z^{2})I - (1-u)(Q-(1-2u)I)z^{2} -
     B_{0}(u)(I-Az+(1-u)(Q+uI)z^{2})\\
     & = Az-2(1-u)(Q+uI)z^{2}.
     \end{align*}
 \end{proof}

 \begin{Lemma}\label{lem:Lemma3} \cite{GILa01}
     Let $f:z\in B_{\varepsilon}\equiv \{z\in\Complessi: |z|<\varepsilon\} \mapsto
     f(z)\in \cb(X)$, be a $C^{1}$- function such that $f(0)=0$ and
     $\|f(z)\|<1$, for all $z\in B_{\varepsilon}$.  Then
     $$
     \t \left( -\frac{d}{dz} \log(I-f(z)) \right) = \t \bigl(
     f'(z)(I-f(z))^{-1} \bigr).
     $$
 \end{Lemma}

 \begin{Cor} \label{Cor:B}
     $$
     \t \left( \sum_{m\geq 1} B_{m}(u)z^{m} \right) =
     \t \left( -z\frac{\partial}{\partial z} \log(I-Az+(1-u)(Q+uI)z^{2})
     \right), \ u\in{\mathcal U},\ |z|<\frac{1}{\alpha}.
     $$
 \end{Cor}
 \begin{proof}
     It follows from Lemma \ref{lem:eq.for.B} $(iv)$ that
     \begin{align*}
     \t \biggl( \sum_{m\geq 1} B_{m}(u)z^{m} \biggr) &=
     \t \bigl(  (Az-2(1-u)(Q+uI)z^{2}) (I-Az+(1-u)(Q+uI)z^{2})^{-1} \bigr)\\
     \intertext{and using the previous lemma with $f(z) :=
     Az-(1-u)(Q+uI)z^{2}$ }
     &= \t \Bigl( -z\frac{\partial}{\partial z} \log(I-Az+(1-u)(Q+uI)z^{2}) \Bigr).
     \end{align*}
 \end{proof}

 We now recall the definition and main properties of the analytic
 determinant on tracial C$^*$-algebras studied in
 \cite{GILa01}

\begin{Dfn} \label{Def:Det}
    Let $(\ca,\tau)$ be a C$^{*}$-algebra endowed with a trace state,
    and consider the subset $\ca_{0}:=\{A\in\ca : 0\not\in
    \conv\sigma(A)\}$, where $\s(A)$ denotes the spectrum of $A$ and
    $\conv\s(A)$ its convex hull.  For any $A\in\ca_{0}$ we set
    $$
    \Det_\t(A)=\exp\, \circ\ \tau\circ\left(\frac{1}{2\pi i}
    \int_\Gamma \log \lambda (\lambda-A)^{-1} d\lambda\right),
    $$
    where $\Gamma$ is the boundary of a connected, simply connected
    region $\Omega$ containing $\conv\sigma(A)$, and $\log$ is a
    branch of the logarithm whose domain contains $\Omega$.
\end{Dfn}

Since two $\G$'s as above are homotopic in $\bc\setminus\conv\s(A)$,
we have

\begin{Cor}\label{cor:det.analytic}
    The determinant function defined above is well-defined and
    analytic on $\ca_{0}$.
\end{Cor}

 We collect several properties of our determinant in the following
 result.

\begin{Prop}
    Let $(\ca,\tau)$ be a C$^{*}$-algebra endowed with a trace state,
    and let $A\in\ca_{0}$.  Then
    
    \item[$(i)$] $\Det_\t(zA)=z\Det_\t(A)$, for any
    $z\in\mathbb{C}\setminus\{0\}$,
     
    \item[$(ii)$] if $A$ is normal, and $A=UH$ is its polar
    decomposition,
    $$
    \Det_\t(A)=\Det_\t(U)\Det_\t(H),
    $$
    
    \item[$(iii)$] if $A$ is positive, then we have
    $\Det_\t(A)=Det(A)$, where the latter is the Fuglede--Kadison
    determinant.
\end{Prop}

 We now recall from \cite{GILa01} the notion of average Euler--Poincar\'e characteristic of a self-similar graph, and from \cite{ChGr} that of $L^{2}$-Euler characteristic of a periodic graph.
 
 \begin{Lemma} \label{lem:EuPoi} \label{Lem:Poincare}
     Let $X$ be a self-similar graph. The following limit exists and is finite: 
     $$
     \chi_{av}(X) := \lim_{n\to\infty} \frac{\chi(K_n)}{|K_n|} =
     -\frac12 \Tr_\cg(Q-I),
     $$
     where $\chi(K_n)=|VK_n| - |EK_n|$ is the Euler--Poincar\'e
     characteristic of the subgraph $K_n$.  The number $\chi_{av}(X)$
     is called the average Euler--Poincar\'e characteristic of the
     self-similar graph $X$.
 \end{Lemma}

\begin{Dfn} 
	Let $(X,\G)$ be a periodic graph. Then  $\displaystyle \chi^{(2)}(X) := \sum_{v\in\cf_0}     \frac{1}{|\G_{v}|} - \frac12 \sum_{e\in \cf_1}   \frac{1}{|\G_{e}|}$ is the $L^{2}$-Euler--Poincar\'e characteristic of $(X,\G)$.
\end{Dfn}

\begin{rem}
	\item[(1)] It was proved in \cite{GILa01} that $\chi_{av}(X)  =  -\frac12 \Tr_\cg(Q-I)$.
	\item[(2)] It is easy to prove that $\chi^{(2)}(X)  =  -\frac12 \Tr_\G(Q-I) = \chi(X/\G)$.
\end{rem}

In the next Theorem we denote by $\chi(X)$ the average or $L^2$- Euler--Poincar\'e characteristic of $X$, as the case may be.

 \begin{Thm}[Determinant formula]\label{Thm:detForm}
 	Let $X$ be a periodic or self-similar graph. Then
     $$
     \frac{1}{Z_{X}(z,u)} = (1-(1-u)^{2}z^2)^{-\chi(X)}
     \Det_{\t} \bigl(I-Az+(1-u)(Q+uI)z^{2} \bigr), \ u\in{\mathcal U},\ |z|<\frac{1}{\alpha}.
     $$
 \end{Thm}
 \begin{proof}
     \begin{align*}
     \t\biggl( \sum_{m\geq 1} B_{m}(u)z^{m} \biggr) &=
     \sum_{m\geq 1} \t (  B_{m}(u) ) z^{m}\\
     \intertext{ (by Lemma \ref{lem:eq.for.B} $(iii))$ }
     &= \sum_{m\geq 1} N_{m}(u)z^{m} - \sum_{k\geq 1}
     (1-u)^{2k} \t (Q-I) z^{2k} \\
     &= \sum_{m\geq 1} N_{m}(u)z^{m} - \t (Q-I)
     \frac{(1-u)^{2}z^{2}}{1-(1-u)^{2}z^{2}}.
     \end{align*}
     Therefore, from Proposition \ref{lem:power.series} and Corollary \ref{Cor:B}, we obtain
     \begin{align*}
     z\frac{\partial}{\partial z} \log Z_{X}(z,u) & = \sum_{m\geq 1}
     N_{m}(u)z^{m} \\
     &= \t\left( -z\frac{\partial}{\partial z}
     \log(I-Az+(1-u)(Q+uI)z^{2}) \right) \\
     & \qquad - \frac{z}{2}\frac{\partial}{\partial z}
     \log(1-(1-u)^{2}z^{2}) \t(Q-I)
     \end{align*}
     so that, dividing by $z$ and integrating from $z=0$ to $z$, we
     get
     $$
     \log Z_{X}(z,u) = - \t\bigl( \log(I-Az+(1-u)(Q+uI)z^{2}) \bigr)
     -\frac12 \t(Q-I) \log(1-(1-u)^{2}z^{2}),
     $$
     which implies that
     $$
     \frac{1}{Z_{X}(z,u)} = (1-(1-u)^{2}z^{2})^{\frac12
     \t (Q-I)} \cdot\exp\circ \t\circ  \log(I-Az+(1-u)(Q+uI)z^{2}),
     $$
     and the thesis follows from Lemma \ref{Lem:Poincare} and Definition \ref{Def:Det}.
 \end{proof}

\section{Functional equations for  infinite graphs}

\subsection{Functional equations for the  Bartholdi zeta function of an infinite graph}
In this subsection, we shall prove that a suitable completion of  the Bartholdi zeta functions for essentially $(q+1)$-regular infinite graphs satisfy a functional equation, where a graph is called essentially $(q+1)$-regular if $\deg(v)=q+1$ for all but a finite number of vertices, and $d\equiv \sup_{v\in VX} \deg(v) = q+1$.  

The completion considered here is the function
\begin{align*}
\xi_{X}(z,u) 
&= (1-(1-u)^2z^2)^{(q-1)/2} (1-(q+1)z+(1-u)(q+u)z^2) Z_{X}(z,u) .
\end{align*}

We shall show that 

\begin{Thm}[Functional equation]\label{FunctEq}
Let $X$ be an essentially $(q+1)$-regular infinite graph, periodic or self-similar, and set
$g(z,u)=\frac{1+(1-u)(q+u)z^2}{z}$. Then
\item[(1)] the function $\xi_{X}$ analytically extends to the complement $\cv$ of the set 
 \[
 \Omega=\{(z,u)\in\bc^2:g(z,u) \in[-d,d]\},
 \]
\item[(2)] the set $\cv_0=\{(z,u)\in\Omega^c:z\neq0, u\neq1,u\neq-q\}$ is invariant w.r.t. the trasformation $\psi:(z,u)\mapsto(\frac1{(1-u)(q+u)z},u)$,
\item[(3)] the analytic extension of $\xi_{X}$  satisfies the functional equation
$$\xi_{X}(z,u)=\xi_{X}\circ\psi(z,u), \quad(z,u)\in\cv_0.
$$
\end{Thm}
   
 \begin{Lemma} \label{OmegaShaped} 
 
 Let $d$ be a positive number, and consider the set
 \[
 \Omega_w=\{z\in\Complessi:\frac{1+wz^2}{z} \in[-d,d]\},\quad w\in\Complessi.
 \]
 Then $\Omega_w$ disconnects the complex plane {\it iff} $w$ is real and $0<w\leq\frac{d^2}4$. 
\end{Lemma}
\begin{proof} 
If $w=0$, $\Omega_w$ consists of the two disjoint half lines $(-\infty,-\frac1d]$, $[\frac1d,\infty)$.
 
 If $w\ne0$, the set $\Omega_w$ is closed and bounded. Moreover, setting $z=x+iy$ and $w=a+ib$,
 the equation  $\textrm{Im } \frac{1+wz^2}{z}=0$ becomes 
 \begin{equation}\label{cubic}
 (x^2+y^2)(ay+bx)-y=0.
 \end{equation}
 
  Let us first consider the case $b=0$. If  $a<0$, (\ref{cubic}) implies $y=0$, therefore $\Omega_w$ is bounded and contained in a line, thus does not disconnect the plane.
  
  If $a>0$, $\Omega_w$ is determined by 
  \begin{align}
 & (a(x^2+y^2)-1)y=0, \label{cubicBis}\\
 &  |x+ax(x^2+y^2)|\leq d(x^2+y^2).\label{condition}
  \end{align}
  
  If $a>\frac{d^2}4$, condition (\ref{condition}) is incompatible with $y=0$, while 
condition (\ref{condition}) and $a(x^2+y^2)-1=0$ give $2|x|\leq \frac{d}{a}$, namely only an upper and a lower portion of the circle $x^2+y^2=\frac1a$ remain, thus the plane is not disconnected.
 
 A simple calculation shows that, when  $0<w\leq\frac{d^2}4$, $\Omega_w$ as a shape similar to $\Omega_q$ in Figure \ref{fig:Omega}.

 
 Let now $b\ne0$. We want to show that the cubic in (\ref{cubic}) is a simple curve, namely is non-degenerate and has no singular points, see  Figure \ref{fig:Curva}. 
 
       \begin{figure}[ht]
    \centering
    \psfig{file=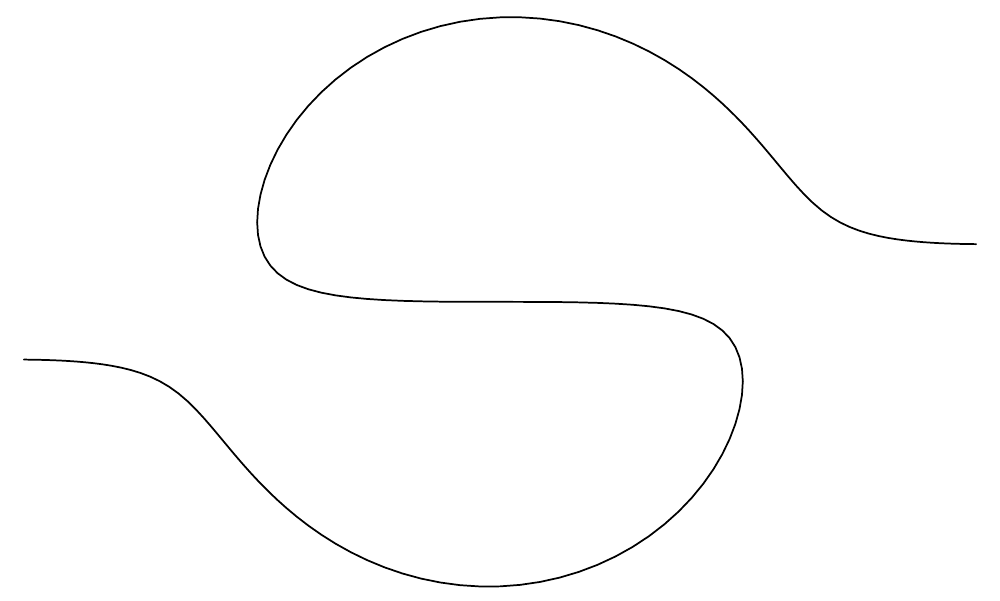,height=1.5in}
    \caption{The cubic containing $\Omega_w$ for $\textrm{Im } w\neq 0$}
    \label{fig:Curva}
      \end{figure}

Up to a rotation, the cubic can be rewritten as
 \[
 (a^2+b^2)(x^2+y^2)y-ay+bx=0.
 \]
 The condition for critical points gives the system
 \[
 \begin{cases}
 (a^2+b^2)(x^2+y^2)y=ay-bx\\
 (a^2+b^2)(x^2+3y^2)=a\\
 2 (a^2+b^2)xy=-b.
 \end{cases}
 \]
 The first two equations give $2(a^2+b^2)y^4=bxy$, which is incompatible with the third equation, namely the cubic curve is simple. Since only a finite portion of the cubic has to be considered, because $\Omega_w$ is bounded, again the plane is not disconnected by $\Omega_w$.
\end{proof}

 \begin{Lemma}\label{lem:connected}
 The set $\Omega$ does not disconnect $\bc^2$. The function $$(z,u) \mapsto \Det_{\t}((1+(1-u)(q+u)z^2)I-zA)$$ is a non-vanishing analytic function on  $\cv=\Omega^c$.
  \end{Lemma}

\begin{proof}
We first observe that the plane $\{0\}\times \bc$ is contained in $\cv$. Set now $T=\{v\in\bc:(1-v)(q+v)\in\br, 0<(1-v)(q+v)\leq\frac{d^2}4\}$.
If $u\not\in T$, any $ z\in\bc:(z,u)\in\cv$ is connected to the point $(0,u)$  by the preceding Lemma. Since $\cv$ is open and $T$ is 1-dimensional, 
for any $ z\in\bc:(z,u)\in\cv$ there exists a ball centered in $(z,u)$ still contained in $\cv$, and such a ball contains a $(z,u')$ with $u'\not\in T$. This proves that $\cv$ is connected. We now prove the second statement. For $(z,u)\in\cv$, the operator $(1+(1-u)(q+u)z^2)I-zA$ is invertible. Indeed this is clearly true for $z=0$ and, for $z\neq0$, it may be written as $\displaystyle -z\left(A-g(z,u)I\right)$, which is invertible since the spectrum of $A$ is contained in the interval $[-d,d]$ of the real line (see \cite{Mohar}, \cite{MoWo}), and the condition $(z,u)\in\cv$ means $g(z,u)\not\in[-d,d]$. Such invertibility implies that the determinant is defined and invertible. Analyticity follows from Corollary \ref{cor:det.analytic}.
\end{proof}

 \begin{proof}[Proof of Theorem \ref{FunctEq}]
      The determinant formula gives
\begin{align*}
 Z_{X}(z,u) &= \frac{(1-(1-u)^2z^2)^{-(q-1)/2}}
    {\Det_{\t}((1+(1-u)(q+u)z^2)I-zA)},\\
\xi_{X}(z,u)&=\frac{(1-(q+1)z+(1-u)(q+u)z^2)}
    {\Det_{\t}((1+(1-u)(q+u)z^2)I-zA)} .
\end{align*}
Now the first statement is a consequence of Lemma \ref{lem:connected}. As for the other two statements, 
we have
\begin{align*}
\xi_{X}(z,u)&=\big(g(z,u)-(q+1)\big)\big( \Det_{\t}(g(z,u)I-A)\big)^{-1},\\
 \Omega&=\{(z,u)\in\bc^2:g(z,u) \in[-d,d]\},
\end{align*}
hence the the results follow by the equality $g\circ\psi=g$.
	
 \end{proof}


\subsection{Functional equations for the Ihara zeta function on infinite graphs}

As discussed in the introduction, the possibility of proving functional equations for the Ihara zeta function on regular infinite graphs, relies on the possibility of extending $Z_{X}$ to a domain $\cu$ which is invariant under the transformation $z\to\frac1{qz}$, and then to check the invariance properties under the mentioned transformation. 

Functional equation may fail for two reasons. The first is that the spectrum of the adjacency operator $A$ may consist of the whole interval $[-d,d]$, and that $Z_{X}$ may be singular in all points of the curve $\O_q$ in Figure \ref{fig:Omega}, so that no analytical extension is possible outside $\O_q$.
The second is more subtle, and was noticed by B. Clair in \cite{Clair}. In one of his examples,  which we describe below in Example \ref{ClairExample}, the completion $\xi_X$ analytically extends to the whole complex plane, but the points $\{1,-1\}$ are ramification points for $Z_{X}$, which lives naturally on a double cover of $\bc$. Then the functional equation makes sense only on the double cover, interchanging the two copies, so that it is false on $\bc$.

\subsubsection{Criteria for analytic extension}

We first make a simple observation.

\begin{Prop}\label{HoleInTheSpectrum}
Let $X$ be an infinite $(q+1)$-regular graph as above. Denote by  $E(\l)$ the spectral family of the adjacency operator $A$, and set $F(\l)=\t(E(\l))$. If $F(\l)$ is constant in a neighborhood of a point $x\in(-2\sqrt q,2\sqrt q)$, then $Z_X$ extends analytically to a domain $\cu$ which is invariant under the transformation $z\to\frac1{qz}$, and a suitable completion of such extension satisfies the functional equation.
\end{Prop}
\begin{proof}
Let $z_\pm$ be the two solutions of the equation $1+qz^2-xz=0$. Since $dF(\lambda)$ vanishes in a neighborhood of $x$, the function
\begin{align*}
\Det_{\t}(&(1+qz^2)I-zA)=\exp\int_{\s(A)}\log(1+qz^2-\l z)\ dF(\l)  
\end{align*}
is analytic in a neighborhood of the points $z_\pm$, hence the singularity region $\{z\in\bc: 1-\l z + qz^2=0, \l\in\sigma(A) \}$ does not disconnect the plane. Its complement $\cu$ is therefore connected and invariant under the transformation $z\to\frac1{qz}$. Consider now the completion
\begin{align*}
\xi_{X}(z) 
&= (1-z^2)^{(q-1)/2} (1-(q+1)z+qz^2) Z_{X}(z)\,.
\end{align*}
The determinant formula in Theorem \ref{Thm:detForm} gives, for $|z|<\frac1q$, $z\neq0$,
\begin{align*}
\xi_{X}(z) 
&=  \big(1-(q+1)z+qz^2\big)  \big(\Det_{\t}((1+qz^2)I-zA)\big)^{-1}\\
&=  \Big(\frac1z+qz-(q+1)\Big)  \Big(\Det_{\t} \Big( (\frac1z+qz)I-A \Big)\Big)^{-1}\\
&=  \Big(\frac1z+qz-(q+1)\Big)\exp\left(-\int_{\s(A)}\log\Big((\frac1z+qz)-\l \Big)\ dF(\l)\right).
\end{align*}
As explained above, $\xi_X$ analytically extends to the region $\cu$, where the functional equation  follows by the invariance of the expression $(\frac1z+qz)$ under the transformation $z\to\frac1{qz}$.
\end{proof}


The following criterion is valid also when there are no holes in the spectrum. A regularity assumption on the spectral measure of $A$ will guarantee that the behaviour of $Z_{X}$ on the critical curve $\O_q$ is not too singular, allowing analytic continuation outside $\O_q$ and the validity of a functional equation for suitable completions.

%
With the notation above,
\begin{align*}
\Det_{\t}(&(1+qz^2)I-zA)=\exp\int_{\s(A)}\log(1+qz^2-\l z)\ dF(\l)  \\
&=(1- (q+1)z+qz^2)\cdot\exp \int_{-d}^{d}\frac{z}{1+qz^2-\l z}F(\l)\ d\l\   ,
\end{align*}
where we used integration by parts and the fact that $\s(A)\subseteq[-d,d]$, hence
\begin{equation}
\xi_{X}(z) = \exp
\left(  - \int_{-d}^d\frac{z}{1+qz^2-\l z}F(\l)\ d\l   \right).
\end{equation}

\begin{Thm}[Functional equation]
Assume there exist $\eps>0$, $\s,\t\in\{-1,1\}$, and a function $\f$ such that 
$\f$ is analytic in $\{\s \Im z>0,|z-2\t\sqrt{q}|<\eps\}$, and $F$ is the boundary value of $\f$ 
on $[2\t\sqrt{q}-\eps,2\t\sqrt{q}+\eps]$,
Then, there exists a connected domain $\cu$ containing $\{|z|<1/q\}$ such that
\begin{itemize}
\item[$(i)$] $\xi_{X}$ extends analytically to $\cu$,
\item[$(ii)$] $\cu\setminus\{0\}$ is invariant under the transformation $z\to\frac1{qz}$,
\item[$(iii)$] the function $\xi_{X}$, extended as above, verifies
$$\xi_{X}(\frac1{qz})=\xi_{X}(z),\quad z\in\cu\setminus\{0\}.$$
\end{itemize}
\end{Thm}

\begin{proof}
We give the proof for $\s=\t=1$, the other cases being analogous. Let $\G$ be the oriented curve in $\bc$ made of the segment $[-d,2\sqrt{q}-\eps]$, the upper semicircle $\{\Im \lambda\geq0,|\lambda-2\sqrt{q}|=\eps\}$, and the segment $[2\sqrt{q}+\eps,d]$. It is not restrictive to assume that $\f$ has a continuous extension to the upper semicircle. We have
$$
\int_{-d}^{d}\frac{z}{1+qz^2-\l z}F(\l)\ d\l - \int_{\G}\frac{z}{1+qz^2-\l z}\f(\l)\ d\l
=\int_{S}\frac{z}{1+qz^2-\l z}\f(\l)\ d\l,
$$
if $S$ is the contour of the semi-disc $D=\{\Im\lambda>0,|\lambda-2\sqrt{q}|<\eps\}$. If $|z|<1/q$ and $\eps$ is small enough, $1+qz^2-\l z$ does not vanish for $\l\in D$, hence the last integrand is analytic, and the contour integral vanishes. As a consequence, for $|z|<1/q$,
\begin{equation}\label{ZetaBis}
\xi_{X}(z) = \exp
\left(  - \int_{\G}\frac{z}{1+qz^2-\l z}\f(\l)\ d\l   \right).
\end{equation}
We now observe that the singularities of the integral are contained in the set $\widetilde\Omega_q:=\{z\in\bc:\frac1z+qz\in\G\}$
shown in Fig. \ref{fig:OmegaBis}, which does not disconnect the plane, and property $(i)$ follows. 

\begin{figure}[ht]
\centering
\psfig{file=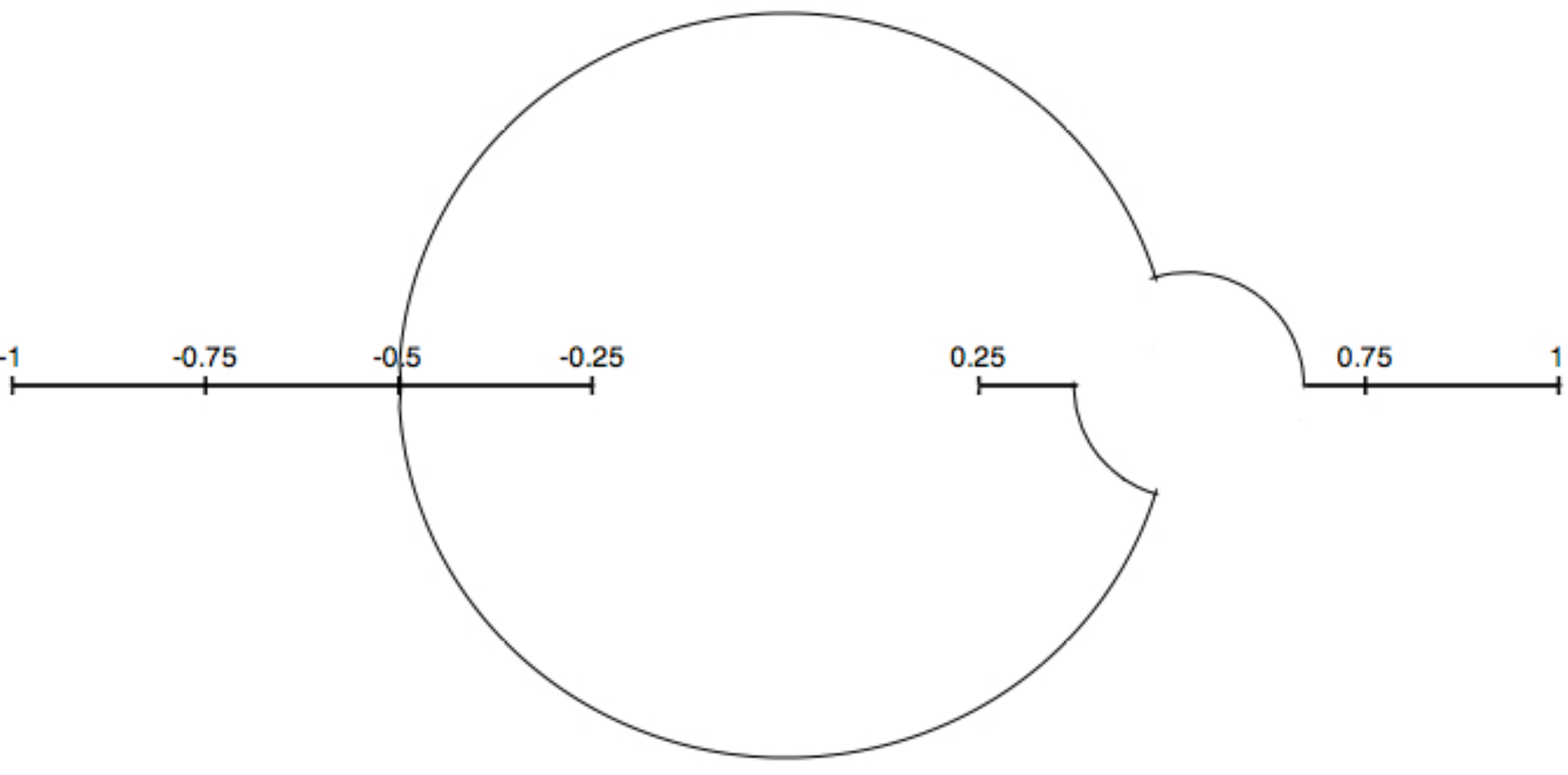,height=2.5in}
\caption{The set $\widetilde\Omega_q$ for $q=4$, $\eps=0.2$}
\label{fig:OmegaBis}
\end{figure}

\noindent
The definition of $\widetilde\Omega_q$ guarantees its invariance under the transformation $z\to\frac1{qz}$, i.e. $(ii)$ is proved. Property $(iii)$ follows directly from  equation (\ref{ZetaBis}).
\end{proof}

\begin{rem}
Let us notice that, when $q=1$, the criterion above is useless. Indeed, since $F(\lambda)$ is constant before $-2$ and after $2$, analyticity implies it should be constant in a neighborhood of either  $-2$ or $2$, namely already Proposition \ref{HoleInTheSpectrum} applies. In particular, the results above do not apply to the Example  \ref{ClairExample}.
\end{rem}

\subsubsection{An extension via Bartholdi zeta function}

The extension we discuss here is again based on analytic extension, but in the sense of two-variable functions. Moreover, it does not require either holes in the spectrum or regularity assumptions on the function $F(\l)$.

As shown above, the Ihara zeta function coincides with the Bartholdi zeta function for $u=0$, $|z|<1/q$, and the latter has a unique analytic extension to the set $\Omega^c$. We may therefore extend the Ihara zeta function via
\begin{equation}\label{ZetaTer}
Z_{X}(z):=Z_{X}(z,0),\qquad (z,0)\in\Omega^c,
\end{equation}
where the Bartholdi zeta function has been extended to $\Omega^c$. Let us remark that $\{z\in\bc:(z,0)\in\Omega^c\}=\Omega_q^c$, cf. Fig. \ref{fig:Omega}.

The following result follows directly by Theorem \ref{FunctEq}.

\begin{Cor}
Assume $X$ is an infinite graph $($either periodic or self-similar$)$, which is essentially $(q+1)$-regular. 
Then, the domain $\Omega_q^c$ contains $\{|z|<1/q\}$ and is invariant under the transformation $z\to\frac1{qz}$. Moreover, setting
\begin{align*}
\xi_{X}(z) 
&= (1-z^2)^{(q-1)/2} (1-(q+1)z+qz^2) Z_{X}(z) .
\end{align*}
where $Z_{X}$ is extended to $\Omega_q^c$ as above, we have
 $\xi(z)=\xi(\frac1{qz})$, for any $z\in\Omega_q^c$.
\end{Cor}


\begin{exmp} \label{ClairExample}
Let us consider the graph $X=\bz$, and the group $\G=\bz$, which acts on $X$ by translations. Using results from \cite{Clair}, we compute the Bartholdi zeta function of $(X,\G)$. We obtain
\begin{align*}
	\frac1{Z_{X,\G}(z,u)} & = \Det_\G(I-Az+(1-u^2)z^2) = \exp \int_{\bt} \Log(1-2\cos\th\, z+(1-u^2)z^2)\,d\th \\
	& = 2z\exp \int_\bt \Log\Big( \frac{z^{-1}+(1-u^2)z}{2} - \cos\th \Big)\, d\th \\
	& = 2z\exp \Big( \arcosh \Big( \frac{z^{-1}+(1-u^2)z}{2} \Big) - \log 2 \Big) \\
	& = z \Big( \frac{z^{-1}+(1-u^2)z}{2} \Big) \Bigg(1+\sqrt{ 1 - \frac{4}{(z^{-1}+(1-u^2)z)^2} }\,\Bigg) \\
	& = \frac{1+(1-u^2)z^2}{2} \Bigg(1+ \sqrt{ 1 - \frac{4z^2}{(1+(1-u^2)z^2)^2} }\,\Bigg)\,,
\end{align*}
which extends to an analytic function on the complement of the set $\O=\{(z,u)\in\bc^2:\frac{1+(1-u^2)z^2}{z}\in[-d,d]\}$.
Therefore the Ihara zeta function, extended via the Bartholdi zeta, is defined on the complement of $\{z\in\bc:|z|=1\}$, where it is given by
\begin{align*}
Z_{X,\G}(z) & = Z_{X,\G}(z,0) = 
\frac2{1+z^2}  \bigg(1+ \sqrt{ \Big( \frac{ 1-z^2 }{ 1+z^2  } \Big)^2 }\, \bigg)^{-1} \\
& = \begin{cases}
\ds \frac2{1+z^2}   \Big( 1+ \frac{ 1-z^2 }{ 1+z^2  } \Big)^{-1} = 1, & |z|<1,\\
\ds \frac2{1+z^2}   \Big( 1- \frac{ 1-z^2 }{ 1+z^2  } \Big)^{-1} = z^{-2}, & |z|>1.
\end{cases}
\end{align*}
and the completion $\xi_{X,\G}$ is given by
$$
\xi_{X,\G}(z)=(z-1)^2Z_{X,\G}(z)=
\begin{cases}
(z-1)^2, & |z|<1,\\
 (\frac1z-1)^2, & |z|>1.
\end{cases}
$$
Defined in this way, $\xi_{X,\G}$ satisfies the functional equation, but its behaviour outside the disc is not given by the analytic extension on $\bc$.	
\end{exmp}


\end{document}